\documentclass[aop,citesort,MSNbibl,dvips]{arximspdf}
\usepackage{graphicx}

\doi{10.1214/10-AOP615}
\volume{40}
\issue{1}
\pubyear{2012}
\firstpage{146}
\lastpage{161}

\makeatletter

\newcommand{\xrightarrow}{\stackrel{n\to\infty}{\hbox to 1cm{\rightarrowfill}}}

\newcommand{\eqref}[1]{(\ref{#1})}

\newcommand{\Q}{{\mathbb Q}}
\newcommand{\Z}{{\mathbb Z}}
\newcommand{\R}{{\mathbb R}}

\newcommand{\PP}{{\mathbb P}}
\newcommand{\E}{{\mathbb E}}

\newcommand{\llbr}{[\![}
\newcommand{\rrbr}{]\!]}
\newcommand{\lp}{(\hspace*{-2pt}|}
\newcommand{\rp}{|\hspace*{-2pt})}
\newcommand{\oo}{\infty}

\newcommand{\sF}{{\mathcal F}}

\newcommand{\es}{\varnothing}
\newcommand{\Om}{\Omega}
\newcommand{\om}{\omega}
\newcommand{\pc}{p_{\mathrm c}}
\newcommand{\Mc}{M_{\mathrm{c}}}

\newcommand{\bc}{\mathbf{c}}

\newtheorem{theorem}{Theorem}
\newtheorem{lemma}[theorem]{Lemma}
\newtheorem{prop}[theorem]{Proposition}

\makeatother

\begin{document}
\begin{frontmatter}

\title{Lattice embeddings in percolation}
\runtitle{Lattice embeddings in percolation}

\begin{aug}
\author[A]{\fnms{Geoffrey R.} \snm{Grimmett}\ead[label=e1]{g.r.grimmett@statslab.cam.ac.uk}\ead[label=u1,url]{http://www.statslab.cam.ac.uk/\textasciitilde grg/}} and
\author[B]{\fnms{Alexander E.} \snm{Holroyd}\corref{}\ead[label=e2]{holroyd@microsoft.com}\ead[label=u2,url]{http://research.microsoft.com/\textasciitilde holroyd}}
\runauthor{G. R. Grimmett and A. E. Holroyd}
\affiliation{Cambridge University and Microsoft Research, and~University~of~British~Columbia}
\address[A]{Statistical Laboratory\\
Centre for Mathematical Sciences\\
Cambridge University\\
Wilberforce Road\\
Cambridge CB3 0WB\\
United Kingdom\\
\printead{e1}\\
\printead{u1}}
\address[B]{Microsoft Research\\
1 Microsoft Way\\
Redmond, Washington 98052\\
USA\\
and\\
Department of Mathematics\\
University of British Columbia\\
Vancouver, BC V6T 1Z2\\
Canada\\
\printead{e2}\\
\printead{u2}}
\end{aug}

\received{\smonth{3} \syear{2010}}
\revised{\smonth{9} \syear{2010}}

%
\begin{abstract}
Does there exist a Lipschitz injection of $\Z^d$ into the open
set of a site percolation process on $\Z^D$, if the percolation
parameter $p$ is sufficiently close to~$1$? We prove a
negative answer when $d=D$ and also when $d\geq2$ if the
Lipschitz constant $M$ is required to be $1$. Earlier work of
Dirr, Dondl, Grimmett, Holroyd and Scheutzow yields a~positive
answer for $d<D$ and $M=2$. As a result, the above question is
answered for all $d$, $D$ and $M$. Our proof in the case $d=D$
uses Tucker's lemma from topological combinatorics, together
with the aforementioned result for $d<D$. One application is
an affirmative answer to a question of Peled concerning
embeddings of random patterns in two and more dimensions.
\end{abstract}

\setattribute{keyword}{AMS}{AMS 2000 subject classification.}
\begin{keyword}[class=AMS]
\kwd{60K35}.
\end{keyword}
\begin{keyword}
\kwd{Lipschitz embedding}
\kwd{lattice}
\kwd{random pattern}
\kwd{percolation}
\kwd{quasi-isometry}.
\end{keyword}

\end{frontmatter}

\section{Introduction and results}
\subsection{Preliminaries}\label{ss1}

Let $\Z^d$ denote the $d$-dimensional integer lattice.
Elements of $\Z^d$ are called \textit{sites}. Let \mbox{$\|\cdot\|_r$}
denote the $\ell^r$-norm on $\Z^d$, and abbreviate
\mbox{$\|\cdot\|_1$} to \mbox{$\|\cdot\|$}. We say that a map
$f\dvtx\Z^d\to\Z^D$ is $M$-\textit{Lipschitz}, or
$M$-\textit{Lip}, if $\|f(x)-f(y)\|\leq M$ for all $x,y\in\Z^d$ with
$\|x-y\|=1$.

For $p \in[0,1]$, consider the site percolation model on
$\Z^D$. That is, declare each site to be \textit{open} (or
$p$-\textit{open}) with probability $p$, and otherwise \textit{closed},
with different sites receiving independent
designations. Let $W_p(\Z^D)$ denote the random set of open
sites, and write $\PP_p$ and $\E_p$ for the associated
probability measure and expectation operator.

We are interested primarily in the probability
%
\begin{equation}\label{h1}
L(d,D,M,p):= \PP_p\bigl(
\mbox{$\exists$ an $M$-Lip injection from $\Z^d$ to $W_p(\Z^D)$}
\bigr).
\end{equation}
Clearly $L$ is increasing in $D$, $M$ and $p$, and decreasing
in $d$. Furthermore, $L$ is $\{0,1\}$-valued, since $\PP_p$ is a
product measure and the event in \eqref{h1} is invariant under
translations of $\Z^D$. We define the critical probability
\[
\pc(d,D,M):=\inf\{p\dvtx L(d,D,M,p)=1\}
\]
and furthermore
\[
\Mc(d,D):=\min\{M\geq1\dvtx\pc(d,D,M)<1\}
\]
(where $\min\es:=\infty$).
That is, $\Mc(d,D)$ is the
smallest $M$ such that, for some $p<1$, there exists, $\PP_p$-a.s., an
injective $M$-Lip map from $\Z^d$ to the open sites of~$\Z^D$.

Note that $L(1,D,M,p)$ is simply the probability that there
exists an open bi-infinite self-avoiding path in the graph with
vertex-set $\Z^D$ and an edge connecting every pair of sites at
$\ell^1$-distance at most $M$. It follows from standard
percolation results that $\pc(1,D,M)$ is the critical
probability for site percolation on this graph (see, e.g.,
\cite{lyons}, proof of Theorem 3.9, for a
proof for arbitrary graphs). Therefore, for $M\geq1$,
\[
\pc(1,D,M) \cases{ =1, &\quad if $D=1$,\cr
\in(0,1), &\quad if $D\geq2$.}
\]
We deduce in particular that $\pc(d,D,M)>0$ for
all $d,D,M\geq1$. The problem of interest is to
determine for which $d$, $D$, $M$ it is the case that
$\pc(d,\break D, M) = 1$.

\subsection{Main result}\label{ss2}

\begin{theorem}\label{main}Let $d,D,M$ be positive integers.
\begin{longlist}[(a)]
%
\item[(a)] For all $d$, we have $\pc(d,d+1,2)<1$, and hence $\Mc
(d,d+1)\le2$.
\item[(b)] For all $D\geq2$, we have $\pc(2,D,1)=1$, and hence
$\Mc(2,D)>1$.
\item[(c)] For all $d\,{\geq}\,2$ and all $M$, we have $\pc(d,d,M)\,{=}\,1$,
and hence \mbox{$\Mc(d,d)\!=\!\oo$}.
\end{longlist}
\end{theorem}

It is an elementary observation that if $d>D$, then
$L(d,D,M,1)=0$ for all $M$, and hence $\Mc(d,D)=\oo$. [To
check this, suppose that $f\dvtx\Z^d \to\Z^D$ is an $M$-Lip
injection, and let $S_n:= \{x\in\Z^d\dvtx\|x\|\le n\}$. Then
$|S_n|$ has order $n^d$, but $|f(S_n)|$ has order at most $n^D$
(\mbox{$<$}$n^d$), in
%
\begin{table}
\tablewidth=200pt
\caption{The values of $\Mc(d,D)$
for $d,D \ge1$}\label{table1}
\begin{tabular*}{\tablewidth}{@{\extracolsep{\fill}}lcccccc@{}}
\hline
& \multicolumn{6}{c@{}}{$\bolds{D}$}\\[-4pt]
& \multicolumn{6}{c@{}}{\hrulefill}\\
$\bolds{d}$&\textbf{1}&\textbf{2}&\textbf{3}&\textbf{4}&\textbf{5}&$\bolds{\ldots}$ \\
\hline
1 & $\infty$ & 1 & 1 & 1 & 1&\ldots\\
2 & $\infty$ & $\infty$ &2 & 2& 2&\ldots \\
3 & $\infty$ & $\infty$ & $\infty$ &2 & 2&\ldots\\
4 & $\infty$ & $\infty$ & $\infty$ & $\infty$ & 2&\ldots\\
5 & $\infty$ & $\infty$ & $\infty$ & $\infty$ & $\infty$& $\smash\ldots
$\\
\vdots& \vdots& \vdots& \vdots& \vdots& $\vdots$ & $\ddots$ \\
\hline
\end{tabular*}
\vspace*{-2pt}
\end{table}
contradiction of the injectivity of $f$.] Therefore, the
above results suffice to determine the values of $\Mc$ for all
$d$, $D$, as summarized in Table \ref{table1}. We note in particular that
%
\begin{equation}\label{mcoo}
\Mc(d,D) < \oo \quad\mbox{if and only if}\quad d<D.
\end{equation}

Theorem \ref{main}(a) is an immediate consequence of a
substantially stronger statement proved in \cite{ddghs}, which
we state next. For $x=(x_1,\ldots,x_{d-1})\in\Z^{d-1}$ and
$z\in\Z$ we write $(x,z):=(x_1,\ldots,x_{d-1},z)\in\Z^d$. Write
$\Z_+:=\Z\cap(0,\infty)$.
\begin{theorem}[(Lipschitz percolation \cite{ddghs})]\label{don}
Let $d\geq2$, and suppose $p>1-(2d)^{-2}$. There exists
$\PP_p$-a.s. a (random) $1$-Lip function $F\dvtx\Z^{d-1}\to\Z_+$
such that for every $x\in\Z^{d-1}$, the site $(x,F(x))\in\Z^d$
is open.\vadjust{\goodbreak}
\end{theorem}

With $F$ given as in Theorem \ref{don}, the map $x\mapsto(x,F(x))$ is
evidently a~$2$-Lip injection, thus establishing Theorem
\ref{main}(a). Other applications of Theorem~\ref{don} appear
in \cite{DDS,g-h-sphere}. Further properties of $F$ are explored in
\cite{GH4}, where an improved bound on the value of $p$
in Theorem \ref{don} is given.

The proof of Theorem \ref{main}(b) is relatively
straightforward and may be found in Section \ref{nn}. (The proof
involves showing that any $1$-Lip injection from $\Z^2$ to the
full lattice $\Z^D$ must satisfy rather rigid conditions.) The
principal contribution of the current paper is Theorem
\ref{main}(c). Interestingly, our proof of this nonexistence
result makes use of the above existence result, Theorem
\ref{don}. Another essential ingredient of this proof is
Tucker's lemma from topological combinatorics (see
\cite{matousek,tucker}).

It is immediate from the definition of $\Mc(d,D)$ that, if
$\Mc(d,D)=\oo$, then $\pc(d,D,M)=1$ for all $M\ge1$. On the
other hand, we have the following result when $\Mc(d,D)<\oo$
[which occurs if and only if $d<D$, as noted in~\eqref{mcoo}
above].
\begin{prop}\label{other0}
Let $d$, $D$ be positive integers such that $\Mc(d,D)<\oo$.
Then $\pc(d,D,M)\to0$ as $M\to\infty$.
\end{prop}

\subsection{Embeddings of patterns}\label{ss3}
The above results concerning maps from~$\Z^d$ to the open sites
of $\Z^D$ have implications in the more general setting of maps
that preserve values indexed by $\Z^d$, as follows. Let $\Om_d
:=\{0,1\}^{\Z^d}$ be the space of percolation configurations,
in which the value $1$ (resp., $0$) is identified with the
state ``open'' (resp., ``closed''). An \textit{embedding} of
a~configuration $\eta\in\Om_d$ into a configuration $\om\in\Om_D$
is an injection $f\dvtx\Z^d\to\Z^D$ such that $\eta(x)=\om(f(x))$
for all $x\in\Z^d$. We call a configuration $\eta\in\Om_d$
\textit{partially periodic} if there exist $x\in\Z^d$ and $r \in
\Z_+$ such that $\eta(x)= \eta(x+ry)$ for all $y \in\Z^d$.
\begin{prop}[(Embedding)]\label{other}
Let $d$, $D$ be positive integers.
\begin{longlist}[(a)]
\item[(a)] Let $d \ge2$, $p\in(0,1)$ and $\eta\in\Om_d$.
For $\PP_p$-a.e. $\om\in\Om_D$,
there exists no $1$-Lip embedding of $\eta$ into $\om$.\vadjust{\goodbreak}
\item[(b)] Let $d<D$.
For every $p\in(0,1)$, there exists $M\geq1$
such that for $\PP_p$-a.e. $\om\in\Om_D$, it is the case
that for every $\eta\in\Om_d$,
there exists an $M$-Lip embedding of $\eta$ into $\omega$.
\item[(c)] Let $d=D$ and let $\eta\in\Om_d$ be a
partially periodic configuration. For every $p\in(0,1)$, $M \ge1$ and for
$\PP_p$-a.e. $\om\in\Om_D$, there exists no $M$-Lip embedding
of $\eta$ into $\om$.\vspace*{-2pt}
\end{longlist}
\end{prop}

The current work was motivated in part by the problem of
Lipschitz embeddings of random one-dimensional configurations
(see \cite{G-demon,glr}). Proposition \ref{other}(a) extends
Theorem \ref{main}(b) to more general configurations than the
all-$1$ configuration. Part (b) answers affirmatively a
question posed by Peled concerning the existence of $M$-Lip
embeddings of $d$-dimensional random configurations into spaces
of higher dimension; see \cite{G-demon}, Section 5. Part (c)
leaves unanswered the question of whether or not there exist $d
\ge1$, $p\in(0,1)$, $\eta\in\Om_d$ and $M <\oo$ such that
with strictly positive probability (and therefore probability
$1$), there exists an $M$-Lip embedding from $\eta$ into a
random configuration $\om\in\Om_d$ having law $\PP_p$.\vspace*{-2pt}

\subsection{Quasi-isometry}\label{ss4}
There is a close connection between the existence of embeddings
and of quasi-isometries. A \textit{quasi-isometry} between two
metric spaces $(X,\delta)$ and $(Y,\rho)$ is a map $f\dvtx X\to Y$
such that there exist constants $c_i\in(0,\oo)$ with:
\begin{longlist}[(a)]
\item[(a)] $\forall x,x'\in X$,
$c_1\delta(x,x')-c_2\leq\rho(f(x),f(x'))\leq
c_3\delta(x,x')+c_4$,
\item[(b)] $\forall y\in Y$, $\exists x\in X$ such
that $\rho(f(x),y)\leq c_5$.
\end{longlist}
We call such $f$ a $\bc$-\textit{quasi-isometry} when we wish to
emphasize the role of the vector $\bc=(c_1,\ldots,c_5)$. It is not
difficult to see that the existence of a~quasi-isometry is a
symmetric relation on metric spaces. Quasi-isometries of
\textit{random} metric spaces are discussed in \cite{peled}. A \textit
{subspace} of a metric space $(X,\delta)$ is a metric space
$(U,\delta)$ with $U\subseteq X$.\vspace*{-2pt}
\begin{prop}[(Quasi-isometry)]\label{quasi}
Let $d$, $D$ be positive integers, and let~$E$ be the event
that there exists a quasi-isometry between $(\Z^d,\ell^1)$ and
some subspace of $(W_p(\Z^D),\ell^1)$.
\begin{longlist}[(a)]
\item[(a)] If $d<D$ then for all $p\in(0,1)$ we have $\PP_p(E)=1$.
\item[(b)] If $d\geq D$ then for all $p\in(0,1)$ we have
$\PP_p(E)=0$.\vspace*{-2pt}
\end{longlist}
\end{prop}

The proofs of Theorem \ref{main}(b), (c) appear, respectively, in
Sections \ref{nn} and \ref{nse}. The remaining propositions are
proved in Section \ref{ebqi}. Section \ref{openp} contains
four open problems.\vspace*{-2pt}

\section{Nearest-neighbor maps}\label{nn}

In this section we prove Theorem \ref{main}(b).
A~(self-avoiding) \textit{path} in $\Z^d$ is a
finite or infinite sequence of distinct sites,
each consecutive pair of which is at $\ell^1$-distance $1$. Let
$e_1,\ldots,e_d\in\Z^d$ be the standard basis vectors of $\Z^d$, and let
$0$ denote the origin.\vadjust{\goodbreak}
\begin{lemma}\label{congruent}
Let $x_1,\ldots,x_k\in\Z^D$ be distinct, and let
$A=A(x_1,\ldots,x_k)$ be the event that there exists a
singly-infinite path $0=y_0,y_1,\ldots$ in $\Z^D$ such that the
sites $(x_i+y_j\dvtx i=1,\ldots,k, j=0,1,\ldots)$ are distinct and
open. If $p<(2D)^{-1/k}$, then $\PP_p(A)=0$.
\end{lemma}
\begin{pf}
Let $A_n$ be the event that there exists a path
$0=y_0,y_1,\ldots,y_n$ of length $n$ in $\Z^D$ such that the
sites $(x_i+y_j\dvtx i=1,\ldots,k, j=0,\ldots,n)$ are
distinct and open. Note that $A$ is the decreasing limit
of $A_n$ as $n\to\infty$. Let~$N_n$ be the number of paths
$0=y_0,\ldots, y_n$ with the properties required for~$A_n$.
Then
\[
\PP_p(A_n)\leq\E_p N_n\leq(2D)^n p^{nk}\xrightarrow0
\qquad\mbox{if }2Dp^k<1.
\]
Here,\vspace*{1pt} $(2D)^n$ is an upper bound for the number of $n$-step
self-avoiding paths $(y_j)$ starting from $0$, while for those
paths for which the sites $x_i+y_j$ are distinct,
$p^{nk}$ is the probability they are all open.
\end{pf}
\begin{pf*}{Proof of Theorem \ref{main}\textup{(b)}}
We must prove that, for any fixed $p<1$ and $D\geq2$, a.s. there
exists no
$1$-Lip injection from $\Z^2$ to $W_p(\Z^D)$.

First, suppose $f$ is a $1$-Lip injection from $\Z^2$ to
the full lattice $\Z^D$, and consider the image of a unit
square. Specifically, take $(i,j)\in\Z^2$ and let
\begin{eqnarray*}
r_1&=&f(i+1,j)-f(i,j),\qquad r_1'=f(i+1,j+1)-f(i,j+1),\\
r_2&=&f(i,j+1)-f(i,j),\qquad r_2'=f(i+1,j+1)-f(i+1,j).
\end{eqnarray*}
Note that the four vectors $r_1$, $r_2$, $r_1'$, $r_2'$ are elements
of $\{\pm e_j\dvtx j=1,\ldots,D\}$ (by the $1$-Lip property); they
satisfy $r_1+r_2'=r_1'+r_2$ (by definition); the pair $r_1$, $r_2$
are neither equal to nor negatives of each other and similarly
for $r_1$, $r_2'$ (a consequence of injectivity). It follows that
$r_1=r_1'$ and $r_2=r_2'$. Since this holds for every unit
square, for any distinct $i,i'\in\Z$, the images under $f$ of the two paths
$\{(i,j)\dvtx j\in\Z\}$ and $\{(i',j)\dvtx j\in\Z\}$ are two disjoint
self-avoiding paths that are translates of each other. (Another
consequence, which we shall not need, is that there exists
$\Delta\subset\{1,\ldots,D\}$ such that all horizontal edges have
images in $\{\pm e_j\dvtx j\in\Delta\}$, and all vertical edges have
images in the complement $\{\pm e_j\dvtx j \notin\Delta\}$.)

Let $B$ be the event that there exist $x_1,x_2,\ldots\in\Z^D$
and a self-avoiding path $0=y_0,y_1,\ldots$ in $\Z^D$ such that
the sites $(x_i+y_j\dvtx i\geq1,j\geq0)$ are distinct and open.
The above argument implies that, if there exists a $1$-Lip
injection $f\dvtx\Z^2\to W_p(\Z^D)$, then $B$ occurs. We shall now
show that $\PP_p(B)=0$ for all $p<1$ and $D\geq1$. Let $k$ be
large enough that $p<(2D)^{-1/k}$. We define $B_k$ analogously
to~$B$, except in that we now require the existence of only $k$
sites $x_1,\ldots,x_k$. Lem\-ma~\ref{congruent} implies that
$\PP_p(B_k)=0$ because $B_k$ is the countable union over all
possible $x_1,\ldots,x_k$ of the events $A(x_1,\ldots,x_k)$.
Finally, we have $B\subseteq B_k$.
\end{pf*}

\section{The case of equal dimensions}\label{nse}

In this section we prove Theorem \ref{main}(c).
We denote integer intervals by $\lp a,b \rrbr:=(a,b]\cap\Z$, et cetera.
Fix any $d\geq2$, $M\geq1$ and $p\in(0,1)$. We will prove that a.s.
there exists no $M$-Lip injection from $\Z^d$ to
$W_p(\Z^d)$.

The idea behind the proof is as follows. Suppose that $f$ is
such an injection. By a \textit{hole} we mean a cube of side
length $M$ in $\Z^d$ all of whose sites are closed (actually, a
slightly different definition will be convenient in the formal
proof, but this suffices for the current informal sketch).
Holes are rare (if $p$ is close to~$1$), but the typical
spacing between them is a fixed function of $d$,~$M$ and~$p$.
We will consider the image under $f$ of a cuboid $\llbr
1,n\rrbr^{d-1}\times\llbr1,m\rrbr\subset\Z^d$, where $m\gg n\gg1$.
We will arrange that the images of the two opposite faces $\llbr
1,n\rrbr^{d-1}\times\{1\}$ and $\llbr1,n\rrbr^{d-1}\times\{m\}$ are
far apart, and separated by a~$(d-1)$-dimensional ``surface of
holes'' (at the typical spacing). This implies that the image of
the interior of the cuboid must pass through this surface,
avoiding all the holes. To do so, the image must be in some
sense be folded up so as to be locally $(d-1)$-dimensional, and
this will give a contradiction to the injectivity of $f$ if $n$
is chosen large enough compared with the spacing of the holes.

In the case $d=2$, it is possible to formalize the above ideas
using fairly direct ad hoc geometric arguments. It is
plausible that a similar approach could be pushed through (with
substantially more difficulty) for $d=3$. For general $d$, a
less direct (but more systematic) approach appears to be
required. Specifically, the surface of holes will be
constructed using Theorem~\ref{don}, and, crucially, we will
augment it with a coloring of the nearby open sites using
exactly $d-1$ colors, in such a way that the colored sites
separate the two sides of the surface from each other, but the
sites of any given color fall into well-separated regions of
bounded size. Via the map $f$, this coloring will induce a
coloring of the cuboid that contradicts a certain topological
fact.

The following notation will be used extensively. A \textit{coloring} of
a set of sites $U\subseteq\Z^d$ is a map $\chi$
from $U$ to a finite set $Q$. A site $u\in U$ is said to have
\textit{color} $\chi(u)\in Q$. We introduce the graph
$G(U,\ell^r,k)$ having vertex set $U$ and an edge between
$u,v\in U$ if and only if $0<\|u-v\|_r\leq k$. An important
special case is the \textit{star-lattice}
$G^*=G^*_d:=G(\Z^d,\ell^\infty,1)$. Given a graph $G$ and a
coloring $\chi$ of its vertex set, a $q$-\textit{cluster} (of
$\chi$ with respect to $G$) is the vertex set of
a connected component in the
subgraph of $G$ induced by the set of vertices of color $q$.
The \textit{volume} of a cluster is defined to be the number of
its sites.

We next state the two main ingredients of the proof: a
topological result on coloring a cuboid and a result on
existence of random colored surfaces in the percolation model.
\begin{prop}[(Color blocking)] \label{tucker}
Let $d,n,m$ be positive integers, and consider a coloring
\[
\chi\dvtx\llbr1,n\rrbr^{d-1}\times\llbr
1,m\rrbr\to\{-\infty,+\infty,1,2,\ldots,d-1\}.
\]
If $\chi$ satisfies:
\begin{longlist}[(a)]
\item[(a)] all sites in $\llbr1,n\rrbr^{d-1}\times\{1\}$ have color
$-\infty$;
\item[(b)] all sites in $\llbr1,n\rrbr^{d-1}\times\{m\}$ have color
$+\infty$; and
\item[(c)] no site of color $+\infty$ is adjacent in $G^*$ to a
site of color $-\infty$,
\end{longlist}
then, for some $j\in\{1,2,\ldots,d-1\}$,
$\chi$ has a $j$-cluster with respect to $G^*$ of volume
at least $n$.
\end{prop}

In fact, Proposition \ref{tucker} remains valid if ``volume'' is
replaced with ``diameter,'' as we shall see.
\begin{prop}[(Colored surfaces)]\label{surface}
Fix $d \geq2$, $J\geq1$ and $p\in(0,1)$. There exist
constants $K,C<\infty$ (depending on $d$, $J$ and $p$) such
that $\PP_p$-a.s. there is a (random) coloring
\[
\lambda\dvtx W_p(\Z^d)\to\{-\infty,+\infty,1,2,\ldots,d-1\}
\]
of the
open sites of $\Z^d$ with the following properties:
\begin{longlist}[(a)]
\item[(a)] No site of color $+\infty$ is adjacent to a site of
color $-\infty$ in $G(W_p(\Z^d)$, $\ell^\infty, J)$.
\item[(b)] For each $j\in\{1,2,\ldots,d-1\}$, every $j$-cluster
with respect to $G(W_p(\Z^d)$, $\ell^\infty,J)$ has
volume at most $K$.
\item[(c)] There exists a (random) nonnegative real-valued function
$g\dvtx\Z^{d-1}\to[0,\infty)$, with the Lipschitz property
that $|g(u)-g(v)|\leq d^{-1} \|u-v\|_1$ for all
$u,v\in\Z^{d-1}$, such that all open sites in
\[
S^-:=\{ (u,z)\dvtx u\in\Z^{d-1}, z < g(u)\}
\]
are colored $-\infty$, while all open sites in
\[
S^+:=\{ (u,z)\dvtx u\in\Z^{d-1}, z>g(u)+C\}
\]
are colored $+\infty$.
\end{longlist}
\end{prop}

In (c) above, note in particular that all
open sites in the half-space $\Z^{d-1}\times\lp{-\infty},0\rp$
are colored $-\infty$.

The proof of Theorem \ref{main}(c) will proceed by playing
Propositions \ref{tucker} and~\ref{surface} against one another
to obtain a contradiction. The number of permitted colors is
crucial---if one color more were added to $1,\ldots,d-1$,
then the conclusion of Proposition \ref{tucker} would no longer
hold, while with one color fewer, the conclusion of
Proposition \ref{surface} would not hold. It should also be
noted that the use of the star-lattice $G^*$ is essential in
Proposition \ref{tucker}---the statement does not hold for
the nearest-neighbor lattice $G(\Z^d,\ell^1,1)$. Another key
point is that the presence of closed sites is essential for
Proposition \ref{surface}---the conclusion does not hold when
$p=1$ (for any $K,C$), since such a coloring would give a
contradiction to Proposition \ref{tucker}.

The choice of the Lipschitz constant $d^{-1}$ in Proposition
\ref{surface}(c) above is relatively unimportant---the result\vadjust{\goodbreak}
would hold for any positive constant, while any constant less
than $(d-1)^{-1}$ would suffice for our application (see Lemma
\ref{intersection} below).

Our proof of Proposition \ref{tucker} will use Tucker's lemma,
a beautiful result of topological combinatorics. The general
version of \cite{lefschetz,tucker} applies to triangulations of
a ball, and is a close relative of the Borsuk--Ulam theorem
(see~\cite{matousek} for background). We need only a special case,
for the cuboid, which is also proved in \cite{baker}.

For $t\in\llbr1,\infty\rp^d$, consider the integer cuboid
$T=T(t):=\llbr0,t_1 \rrbr\times\cdots\times\llbr
0,t_d\rrbr\subset\Z^d$ with opposite corners $0$ and $t$, and
define the boundary $\partial T:=T\setminus[\lp0,t_1 \rp
\times\cdots\times\lp0,t_d\rp]$. We say that boundary sites
$x,y\in\partial T$ are \textit{antipodal} if $x+y=t$.\vspace*{-2pt}
\begin{lemma}[(Tucker's lemma for the cuboid \cite{baker})]\label{tucker-cube}
Let $T\subset\Z^d$ be a
cuboid as above, and suppose $\beta\dvtx T\to\{\pm1,\ldots,\pm d\}$
is a coloring such that for each antipodal pair
$x,y\in\partial T$ we have $\beta(x)=-\beta(y)$. Then there
exist $u,v\in T$ that are adjacent in $G^*$ (and, in fact, that
satisfy $u_i\leq v_i\leq u_i+1$ for all $i$) such that
$\beta(u)=-\beta(v)$.\vspace*{-2pt}
\end{lemma}
\begin{pf*}{Proof of Proposition \ref{tucker}}
Throughout the proof, adjacency and clusters refer to $G^*$.
The ($\ell^\infty$-)\textit{diameter} of a cluster is the maximum
$\ell^\infty$-distance between two of its sites. It suffices to
show that for a coloring $\chi$ satisfying the given
conditions, there is a $j$-cluster of diameter at least $n$ for
some $j\neq\pm\infty$. Suppose that this is false. We will
construct a modified coloring that leads to a contradiction.

First define a coloring $\chi'$ of the larger cuboid $T:=\llbr
0,n+1\rrbr^{d-1}\times\llbr0,m+1\rrbr$ as follows. Let $\chi'$
agree with $\chi$ on $T\setminus\partial T$, except with color
$\infty$ everywhere changed to $d$, and $-\infty$ changed to
$-d$. Color $\partial T$ as follows. For each $i=1,\ldots,d$,
let $\chi'$ assign color $-i$ to the face $\{x\in T\dvtx x_i=0\}$,
and color $+i$ to the antipodal face (this rule creates
conflicts at the intersections of faces; for definiteness
assign such sites the candidate color of smallest absolute
value). Thus $\chi'$ satisfies the condition of Lemma
\ref{tucker-cube} on the boundary.

Now let $\beta$ be the coloring of $T$ obtained by modifying
$\chi'$ as follows. For each $i=1,\ldots,d-1$, recolor with
color $-i$ all $i$-clusters that are adjacent to the face
colored $-i$. Since there were no $i$-clusters of diameter as
large as $n$ in $\chi$, this does not affect the colors on
$\partial T$. Hence Lemma \ref{tucker-cube} applies, so there
are adjacent sites $u,v\in T$ with $\beta(u)=-\beta(v)$, which
contradicts the manner of construction of $\beta$.\vspace*{-2pt}
\end{pf*}

The proof of Proposition \ref{surface} relies on Theorem
\ref{don} concerning Lipschitz surfaces in percolation,
together with the following deterministic fact.\vspace*{-2pt}
\begin{lemma}[(Periodic coloring)]\label{tile}
For any integers $d\geq1$ and $R\geq2d$, there exists a
coloring $\alpha\dvtx\Z^d\to\{0,1,\ldots,d\}$ with the following
properties:
\begin{longlist}[(a)]
\item[(a)] The coloring is periodic with period $R$ in each
dimension; that is, $\alpha(x+Ry)=\alpha(x)$ for all\vadjust{\goodbreak}
$x,y\in\Z^d$.
\item[(b)] For each $j\in\{0,1,\ldots,d\}$, every $j$-cluster with
respect to $G^*$ has volume at most $R^d$.
\item[(c)] The $0$-clusters with respect to $G^*$ are precisely the
cubes $Rx+\llbr-(d-1),(d-1)\rrbr^d$, for $x\in\Z^d$.\vspace*{-3pt}
\end{longlist}
\end{lemma}
\begin{pf}
The construction is illustrated in Figure \ref{slabs}. Define
a \textit{slice} to be any set of sites of the form
%
\begin{figure}

\includegraphics{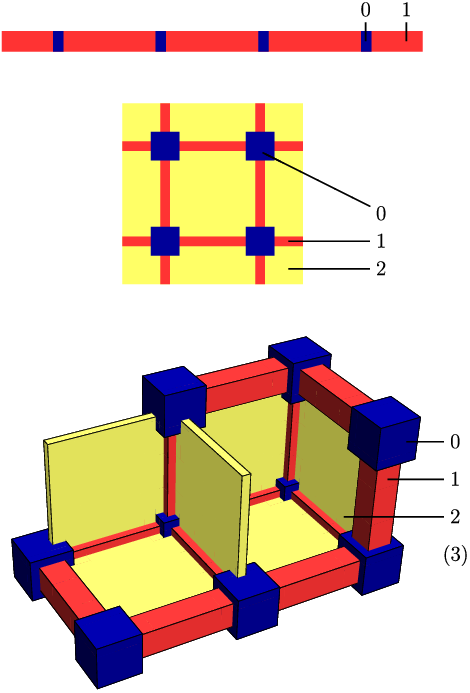}

\caption{Part of the coloring $\alpha$ of Lemma \protect\ref{tile}, for
$d=1$ \textup{(top)}, $d=2$ \textup{(middle)}, $d=3$ (\textup{bottom}; color~$3$ is shown as transparent,
and only selected slabs are shown).}
\label{slabs}\vspace*{-3pt}
\end{figure}
$Y=Rx+(I_1\times\cdots\times I_d)$, where $x\in\Z^d$ and each
$I_i$ is either $\{0\}$ or $\llbr1,R-1\rrbr$. If $\llbr1,R-1\rrbr$
appears $k$ times in this product then we call $Y$ a \mbox{$k$-\textit
{slice}}. The set of all slices forms a partition of~$\Z^d$. Let
$a_k:=d-1-k$. For a $k$-slice $Y$, define the associated
$k$-\textit{slab} to be the set obtained from $Y$ by replacing
each occurrence of $\{0\}$ in the product
$I_1\times\cdots\times I_d$ with $\llbr-a_k,a_k\rrbr$ (thus
``thickening'' the slice by distance $a_k$). We now define the
coloring: for each site $x$, let $\alpha(x)$ be the smallest
$k$ for which~$x$ lies in some $k$-slab.\vadjust{\goodbreak}

The required properties (a) and (c) are immediate [the cubes
in (c) are precisely the $0$-slabs]. For (b), note that any
$k$-cluster is contained within a~single $k$-slab; it is
straightforward to check that, for any given $k$,
any connection in $G^*$ between
two different $k$-slabs is prevented by sites of smaller
colors (here it is important that $a_k$ is strictly decreasing
in $k$). The volume of a $k$-slab is
$(R-1)^k(2a_k+1)^{d-k}< R^d$.\vspace*{-2pt}
\end{pf}

In the following, we sometimes refer to the $d$ coordinate as
vertical, with positive and negative senses being up and down,
%
\begin{figure}[b]

\includegraphics{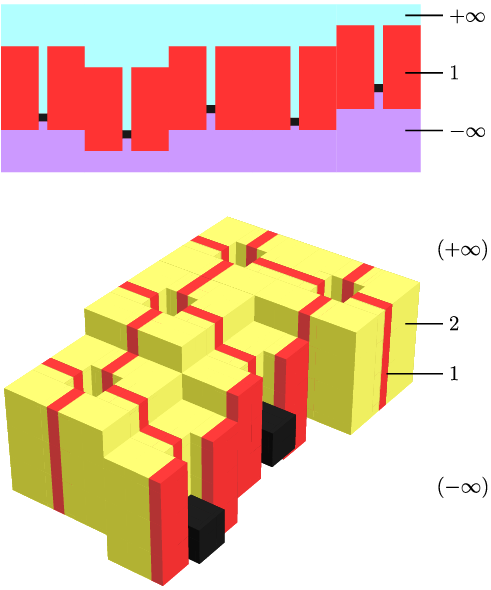}

\caption{Part of the random coloring $\lambda$ of
Proposition \protect\ref{surface}, for $d=2$
\textup{(top)}, and $d=3$ (\textup{bottom}, with colors $\pm\infty$ shown as transparent).
Holes are shown as black.}
\label{fig-surf}\vspace*{-2pt}
\end{figure}
respectively, and the other coordinate directions as
horizontal.\vspace*{-2pt}
\begin{pf*}{Proof of Proposition \ref{surface}}
See Figure \ref{fig-surf} for an illustration of the
construction. Let $L$ be a large constant, a multiple of $J$,
to be determined later, and let $\alpha$ be the coloring from
Lemma \ref{tile} with parameters $d-1$ and $R:=L/J$. Let
$\alpha'$ be the coloring obtaining by dilating $\alpha$ by a
factor $J$; that is, for $u\in\Z^{d-1}$, let
$\alpha'(u)=\alpha([ u/J])$ (where $[v]$ denotes $v$ with each
co-ordinate rounded to the nearest integer, rounding up in the case
of ties).\vadjust{\goodbreak} Note from property (a) in Lemma \ref{tile} that
$\alpha'$ has period $L$ in each dimension, while by (b),
for $j\in\{0,1,\ldots,d-1\}$, each $j$-cluster of $\alpha'$
with respect to $G(\Z^{d-1},\ell^\infty,J)$ has volume at most
$L^{d-1}$. Write $r:=J(2d-3)$. From Lemma~\ref{tile}(c),
the $0$-clusters of $\alpha'$ are $(d-1)$-dimensional cubes of
side length $r$ centred (approximately) at the elements of the
lattice $L\Z^{d-1}$.

For $u\in\Z^{d-1}$ we will use $\alpha'(u)$ to determine the
colors (other than $\pm\infty$) assigned by $\lambda$ to sites
in the vertical column $\{u\}\times\Z$. Colors $1,\ldots,d-1$
will be used unchanged, while color $0$ will be treated in a
different way.

We now introduce a renormalized percolation process, starting
with certain sets to be used in its definition. For a site
$x=(x_1,\ldots, x_d)\in\Z^d$, write
$\underline{x}:=(x_1,\ldots,x_{d-1})\in\Z^{d-1}$ and
$\overline{x}:=x_d$, so that $x=(\underline{x},\overline{x})$.
Let $\underline{C}_{\underline{x}}\subset\Z^{d-1}$ be the
$0$-cluster of $\alpha'$ centred at $L\underline{x}$. Let
$s:=\lfloor L/d\rfloor$ (where $\lfloor\cdot\rfloor$ denotes
the integer part). Let $\overline{C}_{\overline{x}}$ be the
interval $\llbr s\overline{x},s(\overline{x}+1)\rp\subset\Z$.
Define the \textit{cell} corresponding to $x\in\Z^d$ to be the
set of sites $C_x:=\underline{C}_{\underline{x}}\times
\overline{C}_{\overline{x}}$. Thus, each cell is a cuboid of
height $s$, and side length $r$ in each horizontal dimension.
The centers of the cells are spaced at distance $s$ vertically
(so that they abut each other), and at distance $L$
horizontally.

Define a \textit{hole} to be any cube of the form $z+\llbr
1,r\rrbr^d$, where $z\in\Z^d$, all of whose sites are closed in
the percolation configuration. We say that the cell $C_x$ is
\textit{holey} if it contains some hole as a subset. Now we
return to the issue of choosing $L$. Since a hole has volume
$r^d$ (a function of $J$ and $d$), and a cell has height
$s=\lfloor L/d\rfloor$, we may choose $L$ a sufficiently large
multiple of $J$ (depending on $J$, $d$ and $p$) so that the
probability that a cell is holey exceeds $1-(2d)^{-2}$. For
later purposes, ensure also that $L$ is large enough that $s>J$
and $\lfloor(L-r)/2\rfloor>J$. By Theorem \ref{don}, there
exists a.s. a $1$-Lip function $F\dvtx\Z^{d-1}\to\Z_+$, such that
all the cells $C_{(u,F(u))}$ for $u\in\Z^{d-1}$ are holey.\looseness=-1

We specify next a set of sites surrounding each of the
holey cells considered above, to be colored according to
$\alpha'$. For any $\underline{x}\in\Z^{d-1}$, let
$\underline{B}_{\underline{x}}$ be the cube $\{v\in\Z^{d-1}\dvtx
[v/L]=\underline{x}\}$ (so that these cubes partition
$\Z^{d-1}$). Let $\overline{B}_{\overline{x}}$ be the interval
$\llbr s\overline{x},s\overline{x}+L\rp$. Define the \textit{block}
corresponding to $x\in\Z^d$ to be the set of sites
$B_x:=\underline{B}_{\underline{x}}\times
\overline{B}_{\overline{x}}$. Thus $B_x$ is a cube of side $L$
which contains the cell~$C_x$ (at its bottom-center).

Now we define the coloring $\lambda$. For each $u\in
\Z^{d-1}$, call the block $B_{(u,F(u))}$ \textit{active}. To
each open site $y\in B_{(u,F(u))}$, assign the color
$\alpha'(\underline{y})$, provided this is one of the colors
$1,2,\ldots,d-1$. For the remaining sites $y$ in the active
block [those satisfying $\alpha'(\underline{y})=0$],\vspace*{1pt} we proceed
as follows. Since the cell is holey, choose one hole
$H_u\subset C_{(u,F(u))}$. Since the sites in $H_u$ are closed,
they receive no colors. Assign color $\infty$ to all open
sites in the block that lie above the hole $H_u$, and assign
color $-\infty$ to those that lie below $H_u$. (We say that a
site $x$ lies \textit{above} a set $S$ if there exists $y\in S$
with $\underline{x}=\underline{y}$, and for all such $y$ we
have $\overline{x}>\overline{y}$; \textit{below} is defined
similarly with the inequality reversed.) We have assigned
colors to all open sites lying in active blocks. Finally,
assign color $\infty$ to all open sites that lie above some
active block, and color $-\infty$ to all those that lie below
some active block.\vadjust{\goodbreak}

Now we must check that the coloring $\lambda$ has all the
claimed properties. For property (b), note first that if the
function $F$ were constant, then each $j$-cluster for
$j=1,\ldots,d-1$ would have volume at most $L^{d-1}\times
L=L^d$, since the coloring $\alpha'$ has merely been
``thickened'' vertically to thickness $L$. The effect
of taking a nonconstant $F$ is to displace the active blocks
in the vertical direction, and this clearly cannot make these
clusters any larger, so we can take $K=L^d$.

Property (c) follows easily from the Lipschitz property of
$F$. The constant $d^{-1}$ arises because for
$u,v\in\Z^{d-1}$ with $\|u-v\|_1=1$, the centers of the
corresponding blocks are at horizontal displacement $L$ from
each other, and vertical displacement at most $s\leq L/d$. Once
the function $g$ is determined for the centers of the blocks,
it can be defined elsewhere by linear interpolation.

To check property (a), suppose on the contrary that there exist
two sites $x,y$ with respective colors $+\infty,-\infty$
within $\ell^\infty$-distance $J$ of each other. If there is a
single active block such that both $x$ and $y$ lie above, below
or within it, this contradicts the presence of a hole (which
has side length $r>J$) in the corresponding cell. Also, if one
of $x,y$ lies within an active block, then the other cannot lie
above, below or within a different active block, since
$\lfloor(L-r)/2\rfloor>J$. Therefore the only other case to
consider is that $x$ and $y$ lie, respectively, above and below
two different active blocks, say $B_{(u,F(u))}$ and
$B_{(v,F(v))}$, for some $u,v\in\Z^{d-1}$. In this case we must
have $\|u-v\|_\infty=1$ and therefore $|F(u)-F(v)|\leq
\|u-v\|_1\leq d-1$, so the height intervals
$\overline{B}_{F(u)}$ and $\overline{B}_{F(v)}$ overlap by at
least $L-(d-1)s\geq s>J$, giving again a contradiction.
\end{pf*}

To complete the proof of Theorem \ref{main}(c) we will need the
following simple geometric fact in order to find an
appropriate separating surface. For a vector
$x=(x_1,\ldots,x_d)\in\R^d$, write $\widehat{x}_r$ for the
$(d-1)$-vector obtained by dropping the $r$-coordinate.
\begin{lemma}\label{intersection}
Let $a_{\pm1},\ldots,a_{\pm d}$ be positive constants and
define for $i=1,\ldots,d$ the sets
\begin{eqnarray*}
A_i&:=&\{x\in\R^d\dvtx x_i \leq d^{-1} \|\widehat{x}_i\|_1+a_i\}; \\
A_{-i}&:=&\{x\in\R^d\dvtx x_i \geq-d^{-1} \|\widehat{x}_i\|_1-a_{-i}\}.
\end{eqnarray*}
Then $\bigcap_{i=\pm1,\ldots,\pm d} A_i$ is bounded.
\end{lemma}
\begin{pf}
We may assume without loss of generality that the $a_i$ are all
equal, to $a$ say. For $x\in A_i\cap A_{-i}$ we have
$|x_i|\leq d^{-1} \|\widehat{x}_i\|_1+a$, hence for $x$ in
the given intersection, summing the last inequality over $i$
gives
\[
\|x\|_1\leq\frac{d-1}{d}\|x\|_1+da,
\]
hence
$\|x\|_1\leq d^2 a$.\vadjust{\goodbreak}
\end{pf}
\begin{pf*}{Proof of Theorem \ref{main}\textup{(c)}}
Fix $d\geq2$, $M\geq1$ and $p\in(0,1)$, and suppose that $f$
is an $M$-Lip injection from $\Z^d$ to $W_p(\Z^d)$. Let $K$ be
the constant from Proposition \ref{surface} for the given
values of $p$, $d$ and with $J:=dM$. Let $n:=K+1$. Let $N$ be
large enough so that the image $f(\llbr1,n\rrbr^{d-1}\times\{1\})$
is a subset of $\llbr-N,N\rrbr^d$.

Now apply Proposition \ref{surface} (again with parameters $p$,
$d$ and $J=dM$), but to the translated lattice having its
origin at $(N+1)e_d$, to obtain (a.s.) a~coloring of
$W_p(\Z^d)$ in which all open sites in $\Z^{d-1}\times
\lp{-\infty},N\rrbr$ have color $-\infty$. Call this coloring
$\lambda_d$, and let $S^+_d$ be the set corresponding to $S^+$
in Proposition~\ref{surface}(c) (all of whose open sites are
colored $\infty$). Similarly, for each of the two senses of
the $d$ coordinate directions, apply Proposition \ref{surface}
to the lattice rotated and translated so that the part of the
half-axis at distance greater than $N$ from the origin is
mapped to the positive $d$-axis. Thus we obtain $2d$
colorings $\lambda_{i}$ of $W_p(\Z^d)$, with associated sets
$S^+_i$, for $i=\pm1,\ldots,\pm d$, such that all the
colorings assign color $-\infty$ to $\llbr-N,N\rrbr^d$, and
$\lambda_i$ assigns color~$\infty$ to sites sufficiently far
along the $i$ coordinate half-axis.

For each $i$ as above, let $S_i^{++}$ be the set of sites $y$
such that every site within $\ell^1$-distance $dMn$ of $y$ lies
in $S_i^+$. We claim that
\[
Z:=\Z^d\Bigm\backslash \bigcup_{i=\pm1,\ldots,\pm d} S_i^{++}
\]
is a finite set. This follows from Lemma \ref{intersection}
because $\Z^d\setminus S_i^{++}$ lies in a set of the form
$A_i$ in the lemma. [Here it is important that the Lipschitz
constant in Proposition \ref{surface}(c) is $d^{-1}$.]
Since $f$ is injective, it follows that, for some $m>1$, the
site $f((1,\ldots,1,m))$ lies outside $Z$, and hence lies in
$S_I^{++}$ for some $I$. Since $f(\llbr1,n\rrbr^{d-1}\times\{m\})$
has $\ell^1$-diameter at most $dMn$, this implies that $f(\llbr
1,n\rrbr^{d-1}\times\{m\})$ is a subset of $S_I^+$, and is
therefore colored $\infty$ in $\lambda_I$.

Now define a coloring
\[
\chi\dvtx\llbr1,n\rrbr^{d-1}\times\llbr
1,m\rrbr\to\{\infty,-\infty,1,2,\ldots,d-1\}
\]
via $\chi:=\lambda_I\circ f$. By the construction, $\chi$ satisfies
properties (a) and (b) of Proposition \ref{tucker}. Now, if $x,y$ are
adjacent sites in $G^*$ then $\|x-y\|_1\leq d$, and therefore
the $M$-Lip property gives
\[
\|f(x)-f(y)\|_\infty\leq
\|f(x)-f(y)\|_1\leq dM=J,
\]
so $f(x),f(y)$ are adjacent in
$G(W_p(\Z^d),\ell^\infty,J)$. Hence, property (i) in
Proposition \ref{surface} implies that $\chi$ has no two
adjacent sites in $G^*$ with colors $+\infty$ and $-\infty$,
which is property (c) of Proposition \ref{tucker}. Therefore
by Proposition \ref{tucker}, for some $j\neq\pm\infty$, $\chi$
has a $j$-cluster of volume at least $n$ with respect to~$G^*$.
Let~$A$ be such a cluster. Since $f$ is injective, $f(A)$ also
has volume at least~$n$. But by the above observation on
adjacency, $f(A)$ is a subset of some $j$-cluster of
$\lambda_I$ with respect to $G(W_p(\Z^d),\ell^\infty,J)$. This
contradicts property (b) in Proposition \ref{surface} because
$n>K$.
\end{pf*}

\section{Embedding and quasi-isometry}\label{ebqi}

We will use the following simple renormalization construction.
Fix an integer $r\!\geq\!1$. For a site
\mbox{$x\!=\!(x_1,\ldots,x_D)\!\in\!\Z^D$} define
the corresponding \textit{clump} (or $r$-\textit{clump})
to be the
set of $r$ sites given by\looseness=-1
\[
K_x:=\{(x_1,\ldots, x_{D-1},rx_D+i)\dvtx i\in\llbr0,r-1\rrbr\}.\vspace*{-2pt}
\]\looseness=0
The clumps $(K_x\dvtx x\in\Z^D)$ form a partition of
$\Z^D$, with the geometry of $\Z^D$ stretched by a factor $r$
in the $D$th coordinate. If $\|x-y\|=k$, then, for all $u\in
K_x$ and $v\in K_y$, we have $\|u-v\|\leq(2r-1)k$.\vspace*{-3pt}
\begin{pf*}{Proof of Proposition \ref{other0}}
Let $d<D$, which is to say that $\Mc(d,D)<\oo$.
Any given $r$-clump contains one or more open sites with
probability $1-(1-p)^r$. If this probability exceeds $\pc(d,D,M)$,
there exists a.s. an $M$-Lip injection $f\dvtx\Z^d \to\Z^D$
such that, for each $y \in f(\Z^d)$, the clump $K_y$
contains some open site. By choosing one representative open site in
each such clump, we obtain a $(2r-1)M$-Lip injection from $\Z^d$
to $W_p(\Z^D)$. Hence,
\[
\pc\bigl(d,D,(2r-1)M\bigr)\leq1-\bigl(1-\pc(d,D,M)\bigr)^{1/r}.\vspace*{-2pt}
\]
The claim follows by the monotonicity of $\pc$ in $M$.\vspace*{-3pt}
\end{pf*}
\begin{pf*}{Proof of Proposition \ref{other}}
(a) We may assume with loss of generality that $d=2\le D$.
The proof follows that of Theorem \ref{main}(b) as presented
in Section \ref{nn}, without one difference. Let $\eta\in\Om_2$.
The event $A=A(x_1,\ldots,x_k)$
of Lemma \ref{congruent} is redefined as the event that
there exists a singly-infinite path
$0=y_0,y_1,\ldots$ in $\Z^D$ such that the sites
$(x_i+y_j\dvtx i=1,\ldots,k, j=0,1,\ldots)$ are distinct, and
$\om(x_i+y_j) = \eta(i,j)$ for all such $i$, $j$. As in the proof
of Lemma \ref{congruent}, $\PP_p(A)=0$
whenever $\max\{p,1-p\}<(2D)^{-1/k}$. The proof is now
completed as for the earlier theorem.

(b) Let $d<D$ and write $m=\Mc(d,D)<\oo$.
Given $p\in(0,1)$, choose~$r$ sufficiently large that any given
$r$-clump contains both an open and a~closed site with probability
exceeding $\pc(d,D,m)$. There exists a.s. an $m$-Lip
injection $f \dvtx\Z^d \to\Z^D$ such that, for each $y \in
f(\Z^d)$, the $r$-clump $K_y$ contains both an open and a closed
site. Hence, for any configuration $\eta$, by choosing the
open or the closed site as appropriate in each $r$-clump, we obtain
a~$(2r-1)m$-Lip embedding of $\eta$ into $\omega$.

(c) Let $d \ge1$, $M\ge1$, and let $\eta\in\Om_d$ be
partially periodic. Without loss of generality, we may assume,
for some $r \in\Z_+$ and all $y\in\Z^d$, that $\eta(ry) =
\eta(0)=1$. Let $\om\in\Om_d$ and assume there exists an
$M$-Lip embedding $f$ from~$\eta$ into $\om$. Let
$g\dvtx\Z^d\to\Z^d$ be given by $g(x) = f(rx)$. Then $g$ is an
$rM$-Lip injection from $\Z^d$ into $W_p(\Z^d)$. By Theorem
\ref{main}(c), such an injection exists only for $\om$ lying in
some $\PP_p$-null set.\vspace*{-3pt}
\end{pf*}
\begin{pf*}{Proof of Proposition \ref{quasi}}
(a) We assume without loss of generality that $D=d+1$. Given
$p<1$, take $r$ sufficiently large that a given $r$-clump
in~$\Z^D$ contains some open site with probability exceeding
$1-(2D)^{-2}$. By Theorem \ref{don}, there exists a $1$-Lip
map $F\dvtx\Z^{D-1}\to\Z_+$ such that for all $u\in\Z^{D-1}$, the\vadjust{\goodbreak}
$r$-clump $K_{(u,F(u))}$ contains some open site. By choosing
an arbitrary open site to represent each such clump, we obtain
a quasi-isometry of the required form.

(b) Let $d\geq D$, and suppose that with positive probability
there exists a quasi-isometry from $(\Z^d,\ell^1)$ to some
subspace of $(W_p(\Z^D),\ell^1)$. We will prove that, for some
$p'\in(0,1)$ and $M\ge1$, there exists an $M$-Lip injection
$g\dvtx\Z^d \to W_{p'}(\Z^D)$, which will contradict \eqref{mcoo}.

Recall the parameters $\bc=(c_1,c_2,\ldots,c_5)$ in the
definition of a $\bc$-quasi-isometry, and let $Q_\bc$ be the
event that there exists a $\bc$-quasi-isometry from
$(\Z^d,\ell^1)$ to some subset of $(W_p(\Z^D),\ell^1)$. Since
each $Q_\bc$ is invariant under the action of translations of
$\Z^D$, it has probability $0$ or $1$. Under the above
assumption, the event $\bigcup_\bc Q_\bc$ has positive
probability. By the obvious monotonicities in the parameters
$c_i$, this union is equal to the union
$\bigcup_{\bc\in(\Q\cap(0,\infty))^5} Q_\bc$ over rational
parameters, and hence there exists a \textit{deterministic} $\bc$
such that $Q_\bc$ has probability $1$. We choose $\bc$
accordingly, and let~$\sF_\bc$ be the (random) set of
quasi-isometries of the required type.

A quasi-isometry $f \in\sF_\bc$ is not necessarily an
injection, but, by the properties of a $\bc$-quasi-isometry,
there exists $C=C(d,D,\bc)$ such that, for all $y\in W_p(\Z^D)$
we have $|f^{-1}(y)|\leq C$. Let $r=C$, and take $p'\in(0,1)$
sufficiently large that, with probability at least $p$, every
site in any given $r$-clump is $p'$-open. Let $f\in\sF_\bc$ be
such that: for $y\in f(\Z^d)$, every site in $K_y$ is
$p'$-open. Since the pre-image under $f$ of any $y \in\Z^D$
has cardinality $C$ or less, we may construct an injection
$g\dvtx\Z^d \to W_{p'}(\Z^D)$ such that, for $y \in\Z^D$, every
$x\in f^{-1}(y)$ has $g(x)\in K_y$, and furthermore distinct
elements $x\in f^{-1}(y)$ have distinct images $g(x)$. It is
easily seen that $g$ is an $M$-Lip injection for some
$M=M(d,D,\bc)$.\vspace*{-3pt}
\end{pf*}

\section{Open questions}\label{openp}\vspace*{-3pt}

\begin{longlist}[5.4.]
\item[5.1.]
Derive quantitative versions of Theorem \ref{main}. For
example, fix $d,M,p$, and let $N=N(n)$ be the smallest integer
such that there exists an $M$-Lipschitz injection from the cube
$[ 1,n]^d\cap\Z^d$ to the open sites of $[1,N]^d\cap\Z^d$ with
probability at least $\frac12$. How does $N$ behave as
$n\to\infty$?\vspace*{1pt}

\item[5.2.]
For which graphs $G$ and which $M$ is it the case that for $p$
sufficiently close to $1$ there exists an $M$-Lipschitz
injection from $V(G)$ to the open sites of $V(G)$ (where
$M$-Lipschitz refers to the graph metric of $G$)? Theorem
\ref{main}(c) shows that for~$\Z^d$, no $M<\infty$ suffices. On
the other hand, for the $3$-regular tree, $M=2$ suffices, by
the well-known fact that percolation on a 4-ary tree contains a
binary tree for $p$ sufficiently close to $1$.

\item[5.3.]
We may interpolate between $1$-Lipschitz and $2$-Lipschitz maps
as follows. Let $S$ be a subset of $\{-1,0,1\}^D$, and let $G$
be the graph with vertex set $\Z^D$ and an edge between $u,v$
whenever $u-v$ or $v-u$ belongs to $S$. For which $d$, $D$
and $S$ does there exist an injection from $\Z^{d}$ to the open
sites of~$\Z^{D}$ that maps neighbors in~$\Z^d$ to neighbors
in $G$?

\item[5.4.]
Does there exist a configuration $\eta\in\{0,1\}^d$ such that,
with positive probability there exists a\vadjust{\goodbreak} Lipschitz embedding of
$\eta$ into the percolation configuration $\omega$ on $\Z^d$?
When $d=1$, this is related to the main problem
of~\cite{glr}.
\end{longlist}

\section*{Acknowledgments}
We thank Olle H\"aggstr\"om and Peter Winkler for valuable
discussions. G. R. Grimmett acknowledges support from Microsoft Research
during his stay as a Visiting Researcher in the Theory Group in
Redmond. This work was completed during his attendance at a
programme at the Isaac Newton Institute, Cambridge.

%

%
\printaddresses

\end{document}